\documentclass[12pt, oneside]{amsart}
\usepackage{graphicx}
\usepackage{amsfonts}
\usepackage{epsf}
\usepackage{amssymb}
\usepackage{amsmath}
\usepackage{amscd}
\usepackage{amsthm}
\usepackage{tikz}
\usepackage{pdfpages}
\usepackage{fancyhdr}
\usepackage{setspace}
\usepackage{hyperref}
\usepackage[all]{xy}
\usetikzlibrary{matrix}
\usepackage{verbatim}
\usepackage{enumerate}
\usepackage{mathrsfs}
\usepackage{pinlabel}
\usepackage{lscape}
\usepackage{color}
\usepackage[colorinlistoftodos]{todonotes}

\newtheorem{thm}{Theorem}
\newtheorem{cor}[thm]{Corollary}
\newtheorem{theorem}{Theorem}[section]
\newtheorem{proposition}[theorem]{Proposition}
\newtheorem{lemma}[theorem]{Lemma}

\newcommand{\Z}{\mathbb{Z}}

\newcommand{\N}{\mathbb{N}}
\newcommand{\Q}{\mathbb{Q}}

\newcommand{\sign}{\text{sign }}

\newcommand{\im}{\text{Im}}
\newcommand*{\longhookrightarrow}{\ensuremath{\lhook\joinrel\relbar\joinrel\rightarrow}}

\def\co{\colon\thinspace}

\makeatletter
\newtheorem*{rep@theorem}{\rep@title}
\newcommand{\newreptheorem}[2]{%
\newenvironment{rep#1}[1]{%
 \def\rep@title{#2 \ref{##1}}%
 \begin{rep@theorem}}%
 {\end{rep@theorem}}}
\makeatother

\newreptheorem{theorem}{Theorem}
\newreptheorem{lemma}{Lemma}
\newreptheorem{question}{Question}
\newreptheorem{corollary}{Corollary}

\begin{document}
\makeatletter
\providecommand\@dotsep{5}
\makeatother
\rhead{\thepage}
\lhead{\author}
\thispagestyle{empty}


\raggedbottom
\pagenumbering{arabic}
\setcounter{section}{0}


\title{Knot concordance and homology sphere groups}
\author{Paolo Aceto \and Kyle Larson}
\address{Alfr\'ed R\'enyi Institute of Mathematics, Budapest, Hungary}
\email{aceto.paolo@renyi.mta.hu}
\address{Michigan State University, East Lansing, Michigan}
\email{larson@math.msu.edu}

\begin{abstract}
We study two homomorphisms to the rational homology sphere group $\Theta^3_\mathbb{Q}$. 
If $\psi$ denotes the inclusion homomorphism from the integral homology sphere group $\Theta^3_\mathbb{Z}$, then
using work of Lisca we show that the image of $\psi$ intersects trivially with the subgroup of 
$\Theta^3_\mathbb{Q}$ generated by lens spaces. As corollaries this gives a new proof that the cokernel of $\psi$ is infinitely generated, and implies that a connected sum $K$ of 2-bridge knots is concordant to a 
knot with determinant 1 if and only if $K$ is smoothly slice.
Furthermore, if $\beta$ denotes the homomorphism from the knot concordance group $\mathcal{C}$ defined by 
taking double branched covers of knots, 
we prove that the kernel of $\beta$ contains a $\mathbb{Z}^{\infty}$ summand by analyzing  the Tristram-Levine signatures of a family of knots whose double branched covers all bound rational homology balls. 
\end{abstract}
\maketitle


\begin{section}{Introduction}

Some of the most interesting objects in low-dimensional topology are the groups formed from the set of knots under the equivalence of smooth (or topological) concordance, and from the 
set of $R$-homology 
3--spheres under the equivalence of $R$-homology cobordism (for some ring $R$). While progress has been made in answering certain questions about these groups, we remain quite far from a 
satisfactory understanding 
of their structure. Here we focus on the smooth knot concordance group $\mathcal{C}$, the integral homology sphere group $\Theta^3_\mathbb{Z}$, and the rational homology sphere group $\Theta^3_\mathbb{Q}$
(see Section \ref{background} for precise definitions). 

The double branched cover of a knot in $S^3$ is a rational homology sphere, and in fact this operation gives a homomorphism $\beta \co \mathcal{C} \rightarrow \Theta^3_\mathbb{Q}$. 
On the other hand, since a $\mathbb{Z}$-homology sphere (respectively a $\mathbb{Z}$-homology cobordism) is also a $\mathbb{Q}$-homology sphere (respectively a $\mathbb{Q}$-homology cobordism), 
there is an obvious homomorphism $\psi \co \Theta^3_\mathbb{Z} \rightarrow \Theta^3_\mathbb{Q}$ induced by inclusion. In this paper we consider properties of the kernel, image, and cokernel of these two homomorphisms.

In \cite{Lisca2} (following \cite{Lisca1}), Lisca gives a complete description of the subgroup $\mathcal{L}$ of $\Theta^3_\mathbb{Q}$ generated by lens spaces. The image of $\psi$ and $\mathcal{L}$ are special 
subgroups in $\Theta^3_\mathbb{Q}$, and so it is natural to study their intersection.

 \begin{thm}\label{thm:thm3}
 $\psi(\Theta^3_\mathbb{Z})\cap\mathcal{L}=0$.
 \end{thm}
  
In other words, if a connected sum of lens spaces is rational homology cobordant to any integral homology sphere, then it is rationally homology cobordant to $S^3$ and so bounds a rational homology ball.
Conversely, if an integral homology sphere $Y$ is rational homology cobordant to a connected sum of lens spaces, then $Y$ bounds a rational homology ball.

Since the image of $\psi$ intersects trivially with $\mathcal{L}$, we get that $\mathcal{L}$ injects into the cokernel of $\psi$. A simple corollary of Lisca's work is that $\mathcal{L}$ is infinitely generated, 
and so this gives a new proof that the cokernel of $\psi$ is infinitely generated (the first proof appears in \cite{KL} following earlier work of \cite{HLR}). 

\begin{cor}
The cokernel of $\psi$ is infinitely generated. 
\end{cor}

Furthermore, since lens spaces arise as double branched covers of $2$-bridge knots or links, Theorem \ref{thm:thm3} yields the following corollary.

\begin{cor}\label{cor:2bridge}
 Let $K\subset S^3$ be a connected sum of $2$-bridge knots. 
 Assume that $K$ is smoothly concordant to some knot $J$ with $det(J)=1$. Then $K$ is smoothly slice.
\end{cor}

The above result may be relevant in understanding the topological concordance classes of 2-bridge knots.
It is not known whether there exists a 2-bridge (or even alternating) knot which is topologically but not smoothly slice.
One way to exhibit such an example is to find a non-slice 2-bridge knot $K$ that is smoothly concordant to some knot $J$ with trivial
Alexander polynomial. By the work of Freedman (\cite{Freed1}, \cite{Freed2}) $J$ (and thus $K$) is topologically slice. 
Corollary \ref{cor:2bridge} shows that there is no such example.

Fintushel and Stern \cite{FS} showed that the Brieskorn homology sphere $\Sigma(2,3,7)$ bounds a rational homology ball, and since $\Sigma(2,3,7)$ has infinite order in $\Theta^3_\mathbb{Z}$ we get that the kernel 
of $\psi$ is infinite. In fact the argument in \cite{FS} can be modified to show that the figure-eight knot $E$ is rationally slice, that is, $E$ bounds a smoothly embedded disk in a rational homology ball bounded 
by $S^3$ (see Cha \cite{Cha}). It then follows that the manifolds $E_n = S^3_{1/n}(E)$ defined by $1/n$-surgery on $E$ all bound rational homology balls (see Proposition \ref{balls}) and hence are in the kernel of $\psi$ 
(note that $\Sigma(2,3,7) = E_1$). The integral homology sphere $E_n$ has nontrivial Rokhlin invariant when $n$ is odd, and hence this subfamily is nontrivial in $\Theta^3_\mathbb{Z}$.

Each of these homology spheres $E_n$ can also be obtained by taking the double branched cover of a knot, and so these knots belong to the kernel of $\beta$. 
By studying the Tristram-Levine signatures of the family $\{K_n\}_{n\geq0}$ (see Figure \ref{Knstretch}) whose double branch covers are $E_{2n+1}$, we show the following.

\begin{thm}\label{thm:thm1}
Infinitely many of the knots $\{K_n\}$ are linearly independent in $\mathcal{C}$.
\end{thm}

As an immediate corollary we get that the kernel of $\beta$ is infinitely generated. Note that this also follows from combining classical work of 
Casson and Harer \cite{CH} and Litherland \cite{Lith}. In particular, in \cite{CH} it was shown that the family of Brieskorn 
spheres $\Sigma(2,2s-1,2s+1)$ with $s$ odd all bound contractible 4-manifolds, 
and hence are trivial in $\Theta^3_\mathbb{Z}$ and $\Theta^3_\mathbb{Q}$. These manifolds are double branched covers of the torus knots
$T_{2s-1,2s+1}$, and in \cite{Lith} it was shown that torus knots are linearly independent in $\mathcal{C}$.

In fact, a closer inspection of the proof of Theorem \ref{thm:thm1} yields a direct proof of the following slightly stronger result.

\begin{cor}\label{thm:thm2}
The kernel of $\beta$ contains a $\mathbb{Z}^{\infty}$ summand.
\end{cor}


The work of Lisca \cite{Lisca2} shows that $\im(\beta)$ is infinitely generated by considering 
the lens spaces that are double branched covers of 2-bridge knots. 
This also follows from \cite{KL} and \cite{HLR}, but in fact their work shows more. As we point out in Section 
\ref{background}, the end result of these 
two papers is a family of rational homology spheres $\{M_n\}$ that are linearly independent in $\Theta^3_\mathbb{Q}$, such that each $M_n$ is the double branched cover of a knot and 
none are rational homology cobordant 
to an integral homology sphere. Therefore the family $\{M_n\}$ shows that the group $\im(\beta)/(\im(\beta)\cap \im(\psi))$ is infinitely generated.

The paper is organized as follows. In Section \ref{background} we discuss related work to put this paper into context. 
 We provide basic terminology and state a crucial lemma on integral lattices and lens spaces from \cite{Lisca2} in Section \ref{prelattices}, and
  in Section \ref{lattices} we prove Theorem \ref{thm:thm3} and Corollary \ref{cor:2bridge}.
Finally, we prove our results
relating to the homomorphism $\beta$ in Section \ref{beta}. 

\subsection*{Acknowledgements} The authors would like to thank Marco Golla, Ahmad Issa,  \c{C}a\u{g}r{\i} Karakurt, Andr\'as Stipsicz, and Sa\v{s}o Strle for helpful and encouraging conversations,
as well as Daniele Celoria for computer help. The authors also thank the referee for helpful suggestions.
The authors were supported by the ERC Advanced Grant LDTBud.
\end{section}
\begin{section}{Background}\label{background}

Two knots $K_0$ and $K_1$ are smoothly concordant if there exists a smoothly and properly embedded cylinder $C$ in 
$S^3 \times I$ such that 
$C\cap (S^3 \times \{0\})$ is isotopic to $K_0$ and $C \cap (S^3 \times \{1\})$ is isotopic to $K_1$
(equivalently if the connected sum $-K_0\sharp K_1$ is smoothly slice, where $-K_0$ denotes the mirror of $K_0$ with opposite orientation). 
This defines an equivalence relation, and the set of knots under this equivalence relation forms 
an abelian group $\mathcal{C}$ called the smooth knot concordance group, where the group operation is connected sum. The study of the knot concordance group has a long and rich history, and we refer the reader to the 
introduction of \cite{HKL} for a brief survey of known results.

In a different direction, given a ring $R$, the set of $R$-homology 3-spheres under the equivalence relation of $R$-homology cobordism\footnote{A 3--manifold $Y$ is an $R$-homology sphere if its homology with 
$R$ coefficients is the same as that of $S^3$. A cobordism $W$ between 3--manifolds $Y_1$ and $Y_2$ is an $R$-homology cobordism if the induced inclusion maps $\iota_i \co H_*(Y_i; R) \rightarrow H_*(W; R)$ are 
isomorphisms.} forms an abelian group where the operation is connected sum. The most commonly studied groups are those corresponding to $R = \mathbb{Z},\mathbb{Q}$, or $\mathbb{Z}/2\mathbb{Z}$, and the resulting 
groups are denoted $\Theta^3_\mathbb{Z}, \Theta^3_\mathbb{Q}$, or $\Theta^3_{\mathbb{Z}/2\mathbb{Z}}$.

A few properties of the integral homology sphere group $\Theta^3_\mathbb{Z}$ are listed below.

\begin{itemize}

\item The Rokhlin invariant of a homology 3-sphere provides a homomorphism $\Theta^3_\mathbb{Z} \rightarrow \mathbb{Z}/{2\mathbb{Z}}$. Hence any $\mathbb{Z}$-homology sphere with nontrivial Rokhlin invariant is 
nontrivial in $\Theta^3_\mathbb{Z}$. Furthermore, the Ozsv{\'a}th-Szab{\'o} correction terms \cite{OSz} provide a homomorphism $\Theta^3_\mathbb{Z} \rightarrow \mathbb{Z}$, and so any $\mathbb{Z}$-homology sphere 
with nonzero correction term (e.g. the Poincar\'{e} sphere $\Sigma(2,3,5)$) has infinite order in $\Theta^3_\mathbb{Z}$.

\item In \cite{Fur2} Furuta showed that $\Theta^3_\mathbb{Z}$ contains a subgroup isomorphic to $\mathbb{Z}^\infty$ and hence is infinitely generated (see also \cite{FS2}).

\item More recently Manolescu \cite{Man1} disproved the triangulation conjecture (following earlier work of Galewski and Stern \cite{GalStern} and Matumoto \cite{Mat}) by showing that $\Theta^3_\mathbb{Z}$ contains no 
elements of order 2 with nontrivial Rokhlin invariant.
\end{itemize}

The last point illustrates the interest and difficulty in the open question of whether there is torsion in $\Theta^3_\mathbb{Z}$.

Among results regarding $\Theta^3_\mathbb{Q}$ 
we have Lisca's complete description of the subgroup $\mathcal{L}$ of $\Theta^3_\mathbb{Q}$ generated by lens spaces (see \cite{Lisca1} and \cite{Lisca2}).
Using a lattice-theoretic obstruction based on Donaldson's theorem \cite{Don}, Lisca determined all the relations in $\mathcal{L}$ by showing exactly which connected sums of lens spaces bound rational homology balls.

Kim and Livingston \cite{KL} showed that the cokernel of the map $\psi$ defined in the introduction is infinitely generated by strengthening a result due to Hedden, Livingston, and Ruberman \cite{HLR}. In \cite{HLR} a 
linearly independent family of topologically slice knots was constructed, none of which are concordant to a knot with Alexander polynomial equal to 1. The double branched covers of these knots were shown to be linear 
independent in $\Theta^3_{\mathbb{Z}/2\mathbb{Z}}/\im(\psi)$, and in \cite{KL} it was shown that these manifolds are in fact linear independent in $\Theta^3_\mathbb{Q}/\im(\psi)$. Note that the main theorem of \cite{KL} 
is also interesting here. The authors studied the cokernel of the map $\bigoplus \Theta^3_{\mathbb{Z}[1/p]} \rightarrow \Theta^3_\mathbb{Q}$, where the direct sum is over all primes, and showed that the cokernel contains 
a free subgroup of infinite rank as well as infinite subgroups generated by elements of order two or 3--manifolds that bound topological rational homology balls.

Finally, we remark that $\ker(\psi)$ and $\im(\psi)$ remain quite mysterious. In fact, at present it appears the only thing that can be said about the structure of these groups is that $\ker(\psi)$ contains a subgroup 
isomorphic to $\mathbb{Z}$, $\im(\psi)$ contains a $\mathbb{Z}$ summand, and at least one of these is infinitely generated.
\end{section}

\begin{section}{Preliminaries on integral lattices and lens spaces}\label{prelattices}
An \emph{integral lattice} is a pair $(G,Q)$ where $G$ is a finitely generated free abelian group equipped with a $\Z$-valued symmetric bilinear form $Q:G\times G\rightarrow \Z$.
A \emph{morphism} of integral lattices is a homomorphism of abelian groups which preserves the bilinear forms. An \emph{embedding} of integral lattices is an injective morphism.
The standard terminology for bilinear forms applies in the context of integral lattices in the obvious way (e.g. positive definiteness, unimodularity, and so on).
 \begin{lemma}\label{unimodular}
  Let $(G,U)$ be a positive definite unimodular lattice of rank $n$. The following conditions are equivalent.
  \begin{enumerate}
   \item $(G,U)$ is isomorphic to the standard lattice $(\Z^n,I)$;
   \item there exists an embedding $(G,U)\hookrightarrow (\Z^n,I)$;
   \item for some $m\in\N$ there exists an embedding $(G,U)\hookrightarrow (\mathbb{Z}^m,I)$.
  \end{enumerate}
 \end{lemma}
\begin{proof}
 Clearly we only need to prove that $(3)\Rightarrow (1)$. Suppose we are given an embedding $(G,U)\hookrightarrow (\Z^m,I)$. Since this map is injective we must have $m\geq n$.
 Since $(G,U)$ is unimodular we obtain a decomposition
 $$
 (G,U)\oplus (G,U)^{\bot}\cong (\Z^m,I).
 $$
 Any element in the standard basis of $(\Z^m,I)$ must project to zero in either $(G,U)$ or $(G,U)^{\bot}$. This shows that $(G,U)\cong (\Z^n,I)$ and that $(G,U)^{\bot}\cong(\Z^{m-n},I)$.  
\end{proof}
For any pair of coprime integers $(p,q)$ with $p>q>0$, the \emph{lens space} $L(p,q)$ is, by definition, the result of $-\frac{p}{q}$-surgery on the unknot. Each lens space
$L(p,q)$ arises as the boundary of the \emph{canonical plumbed} 4-manifold $P(p,q)$ via the plumbing graph
   \[
  \begin{tikzpicture}[xscale=1.5,yscale=-0.5]
    \node (A0_1) at (1, 0) {$-a_1$};
    \node (A0_2) at (2, 0) {$-a_2$};
    \node (A0_4) at (4, 0) {$-a_n$};
    \node (A1_0) at (0, 1) {$\Gamma_{p,q}:=$};
    \node (A1_1) at (1, 1) {$\bullet$};
    \node (A1_2) at (2, 1) {$\bullet$};
    \node (A1_3) at (3, 1) {$\dots$};
    \node (A1_4) at (4, 1) {$\bullet$};
    \path (A1_2) edge [-] node [auto] {$\scriptstyle{}$} (A1_3);
    \path (A1_3) edge [-] node [auto] {$\scriptstyle{}$} (A1_4);
    \path (A1_1) edge [-] node [auto] {$\scriptstyle{}$} (A1_2);
  \end{tikzpicture}
  \]

where the $a_i$'s are uniquely determined by $a_i\geq 2$ and
$$
[a_1,\dots,a_n]^-:=a_1-\frac{1}{a_2-\frac{1}{\dots}}=\frac{p}{q}.
$$
The second integral homology group of $P(p,q)$ equipped with its intersection form is a negative definite integral lattice which we denote by $(\Z\Gamma_{p,q},Q_{p,q})$ 
and call the \emph{integral lattice associated with}
$L(p,q)$. Since $-L(p,q)\cong L(p,p-q)$ we also obtain a \emph{dual} negative definite integral lattice $(\Z\Gamma_{p,p-q},Q_{p,p-q})$ associated with $L(p,p-q)$. 

Most of this terminology extends to connected sums of 
lens spaces. Given such a connected sum $\sharp_{i=1}^nL(p_i,q_i)$, the canonical plumbing graph is obtained by taking the disjoint union of the canonical plumbing 
graphs associated to each summand. 
The associated integral lattice is
$$
(\Z\Gamma_{p_1,q_1;\dots;p_n,q_n},Q_{p_1,q_1;\dots;p_n,q_n}):=\bigoplus_{i=1}^n (\Z\Gamma_{p_i,q_i},Q_{p_i,q_i}).
$$
Similarly we may extend the notion of dual lattice. The following crucial lemma is implicit in \cite{Lisca2}.
\begin{lemma}\label{Lisca}
  Let $\sharp_{i=1}^nL(p_i,q_i)$ be a a connected sum of lens spaces. Assume that the integral lattice $(\Z\Gamma_{p_1,q_1;\dots;p_n,q_n},Q_{p_1,q_1;\dots;p_n,q_n})$ 
  embeds in the standard unimodular negative definite lattice of the same rank and the same holds for its dual lattice.
 Then, either one summand or the connected sum of two summands smoothly bounds a rational homology ball.
\end{lemma}

\end{section}
\begin{section}{Homology spheres and lens spaces}\label{lattices}

The following is a more precise statement of Theorem \ref{thm:thm3}.

 \begin{theorem}\label{main}
  Let $\sharp_{i=1}^nL(p_i,q_i)$ be a connected sum of lens spaces. The following are equivalent.
  \begin{enumerate}
   \item $\sharp_{i=1}^nL(p_i,q_i)$ smoothly bounds a rational homology ball;
   \item $\sharp_{i=1}^nL(p_i,q_i)$ is $\Q$-homology cobordant to some integral homology sphere.
  \end{enumerate}
 \end{theorem}
\begin{proof}
 Clearly we only need to prove that $(2)\Rightarrow (1)$. 
 Let $P$ be the canonical plumbing associated with $\sharp_{i=1}^nL(p_i,q_i)$ and let $W$ be a rational homology cobordism such that
 $\partial W=- \sharp_{i=1}^nL(p_i,q_i)\sqcup Y$, where $Y$ is some integral homology sphere. 
 Consider the smooth 4-manifold $X:=P\cup_{\partial} W$ whose boundary is $Y$. It is easy to see 
 that this 4-manifold has a negative definite unimodular intersection form. The second integral homology group together with this intersection form is a unimodular 
 negative definite integral lattice $(G,U)$ associated with $Y$. 
 Similarily, let $P^*$ be the canonical plumbing associated with $-\sharp_{i=1}^nL(p_i,q_i)=\sharp_{i=1}^nL(p_i,p_i-q_i)$ and define 
 $X^*:=P^*\cup_{\partial} (-W)$ so that $\partial X=Y=-\partial X^*$. The smooth, closed 4-manifold $X\cup_{\partial} X^*$
 is negative definite and by Donaldson's theorem has standard intersection form. In particular, by lemma \ref{unimodular},
 the unimodular lattice $(G,U)$ is standard (consider the inclusion $X\hookrightarrow X\cup_{\partial} X^*$). 
 It follows that the integral lattice associated with $\sharp_{i=1}^nL(p_i,q_i)$ embeds in the standard negative definite lattice of 
 the same rank via the inclusion $P\hookrightarrow X$. 
 The same argument shows that the integral lattice associated with $\sharp_{i=1}^nL(p_i,p_i-q_i)$ embeds in the standard negative definite lattice of 
 the same rank via the inclusion $P^*\hookrightarrow X^*$.
 
By Lemma \ref{Lisca}, up to rational homology cobordism we can remove either one or two summands from our connected sum of lens spaces. Now we can iterate the whole argument 
with this reduced connected sum. Eventually we will get an empty sum (i.e. $S^3$) and the conclusion follows.

 \end{proof}

\begin{repcorollary}{cor:2bridge}
 Let $K\subset S^3$ be a connected sum of 2-bridge knots. 
 Assume that $K$ is smoothly concordant to some knot $J$ with $det(J)=1$. Then, $K$ is smoothly slice.
\end{repcorollary}
\begin{proof}
 Concordant knots have rational homology cobordant branched double covers (see Proposition \ref{homo}). Moreover, the determinant of a knot is the order of the first integral homology group
 of its branched double cover. Recall that the branched double cover of a $2$-bridge knot is a lens space.  
 It follows that the connected sum of lens spaces (the branched double cover of $K$, say $\Sigma(K)$) is rational homology cobordant to some integral homology sphere
 (the branched double cover of $J$). By Theorem \ref{main}, $\Sigma(K)$ bounds a rational ball and, by \cite{Lisca1} and \cite{Lisca2}, $K$ is smoothly slice.
\end{proof}

\end{section}
\begin{section}{The double branched cover homomorphism}\label{beta}
\begin{figure}
\includegraphics[scale=01.00]{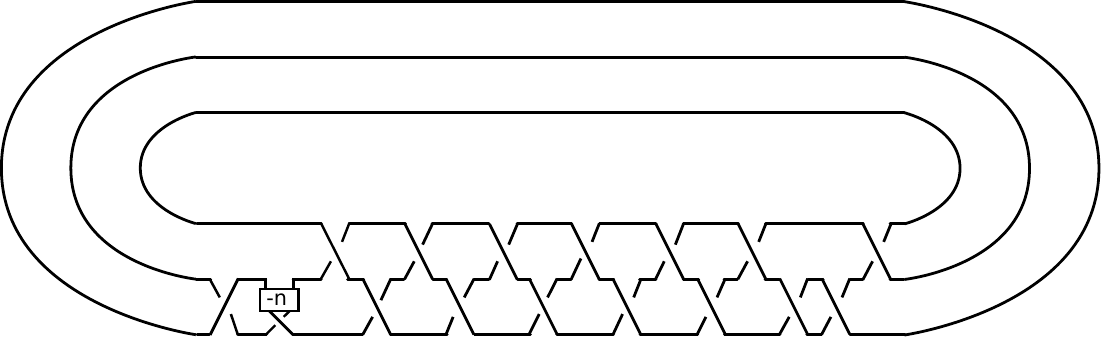}
\centering
 \caption{The knots $K_n$. The box represents $n$ full left-handed twists. Note that $K_0$ is the torus knot $T_{3,7}$.}
\label{Knstretch}
\end{figure}

Let $K$ be a knot in $S^3$, and let $\Sigma_K$ denote its double branched cover.
The following familiar proposition establishes that the assignment $K \mapsto \Sigma_K$ 
induces a well-defined homomorphism $\beta \co \mathcal{C} \rightarrow \Theta^3_\mathbb{Q}$.

\begin{proposition}\label{homo}
 If $K$ is slice, then $\Sigma_K$ bounds a rational homology ball. 
 Furthermore, $\Sigma_{K\sharp K'} = \Sigma_K \sharp \Sigma_{K'}$.
\end{proposition}
\begin{proof}
 For the first assertion see for instance \cite{CG}. 
 The second assertion follows from the fact that the double cover of a 2-sphere branched over two points 
 is again a 2-sphere.
\end{proof}

Now recall the manifolds $E_n = S^3_{1/n}(E)$ defined in the introduction, where $E$ is the figure-eight knot.

\begin{proposition}\label{balls}
The manifolds $E_{n}$ all bound rational homology balls.
\end{proposition}

\begin{proof}
In \cite{Cha} Cha gives an argument (attributed to Cochran) showing how the work of \cite{FS} can be modified to show that the figure-eight knot $E$ bounds a smooth disk $D$ in a rational homology ball $B$. 
A neighborhood $\nu D$ of $D$ is diffeomorphic to $D^2 \times D$ and can be thought of as a 2-handle attached to $B \setminus \nu D$, where the slice disk $D$ is the cocore of the 2-handle (and so $E$ is the belt sphere). 
If we remove this 2-handle and reattach it with a framing that differs from the original framing by $n$ left-handed twists, then the boundary will change by $1/n$-surgery on $E$ (if $n$ is negative we change the 
framing by $|n|$ right-handed twists). Since we are just removing a 2-handle and then reattaching it with a different framing, we are not changing the homology of the 4--manifolds, and so the resulting 3--manifolds 
$E_n$ all bound rational homology balls. (See \cite{LM} for more exposition of this operation, where it was called 2-handle surgery.)
\end{proof}

 It may be worth noting that the rational homology balls constructed in Proposition \ref{balls} have $2$-torsion in their 
 second homology groups, and if $n$ is odd they are necessarily non-spin because
 they have nontrivial Rokhlin invariant. 

Montesinos \cite{Mont} showed that any 3--manifold obtained by Dehn surgery on a strongly-invertible knot can be realized as the double branched cover of a knot or 2-component link in $S^3$.
Since the figure-eight knot $E$ is strongly invertible, we apply his argument to show the following.

\begin{figure}[h!]
\includegraphics[scale=1.20]{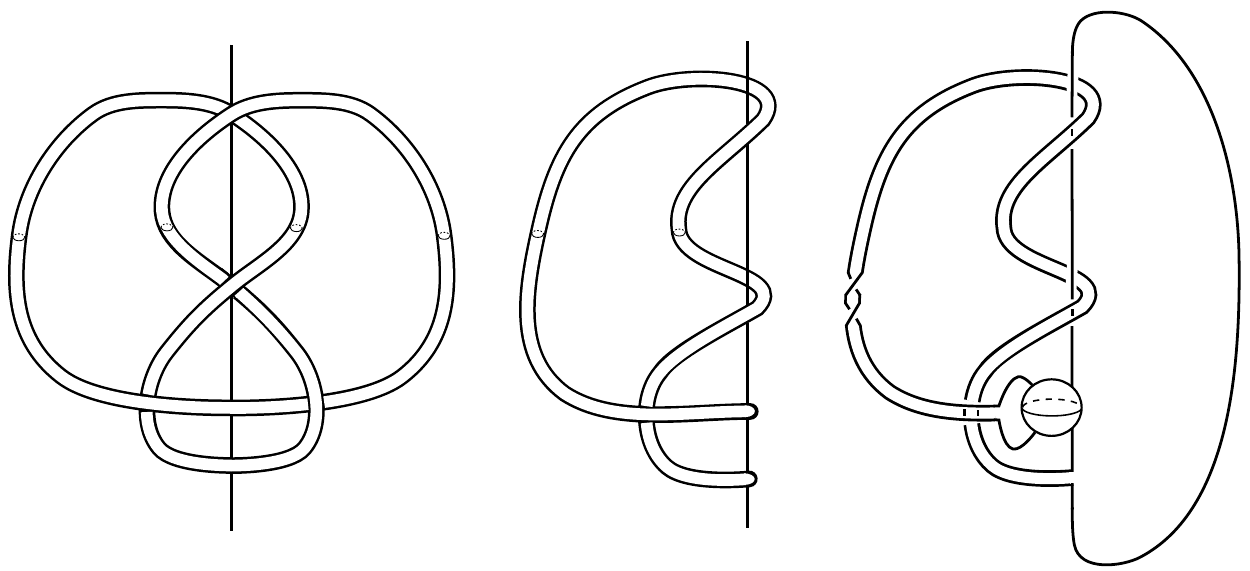}
\centering
 \caption{The involution on the figure-eight knot complement.}
\label{mont}
\end{figure}

\begin{figure}[h!]
\includegraphics[scale=1.90]{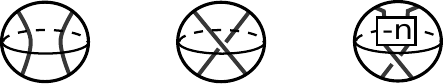}
\centering
 \caption{Rational tangles corresponding to $\infty,+1$, and $\frac{1}{2n+1}$-surgeries.}
\label{trick}
\end{figure}

\begin{proposition}\label{covers}
The manifolds $E_{2n+1}$ are double branched covers of the knots $K_n$ of Figure \ref{Knstretch}.
\end{proposition}

\begin{proof}
The left-hand picture of Figure \ref{mont} shows an involution of $S^3$ (rotation by 180 degrees about the specified axis) taking $E$ to $-E$ with two fixed points.
In fact, this picture shows that there is an induced involution on the knot complement $X_E := \overline{S^3 \setminus \nu E}$.
Now $E_n$ is formed by gluing in a solid torus $U$ to $X_E$, and Montesinos \cite{Mont} showed that we can always arrange for the involution on the knot complement to fit together
with the hyperelliptic involution on $U$ to give an involution on the manifold obtained by the Dehn surgery.
The quotient of $U$ by the hyperelliptic involution is a copy of $B^3$, and in the middle picture of Figure \ref{mont} we see that the quotient of $X_E$ gives $\overline{S^3 \setminus B^3} = B^3$.
The upshot is that $E_n$ is a double branched cover of $B^3\cup B^3 =S^3$, where the branch set is some rational tangle filling of the right-hand picture of Figure \ref{mont}
(to obtain the right-hand picture from the middle picture we isotope the deleted 3-ball, dragging along the arc).

The appropriate fillings are illustrated in Figure \ref{trick}, where the left-hand picture corresponds to $S^3$, the middle picture corresponds to $E_1 = \Sigma(2,3,7)$, 
and the right-hand picture corresponds to $E_{2n+1}$ (we refer to \cite{Auk} for more details on this procedure). 
Lastly, one can check that substituting the tangle in the right-hand picture of Figure \ref{trick} into the right-hand picture 
of Figure \ref{mont} and performing an isotopy results in the knots $K_n$ of Figure \ref{Knstretch}.

\end{proof}

For a knot $K$ in $S^3$, with Seifert form $S$, we define the Alexander polynomial as $\Delta_K(x):=\det(xS-x^{-1}S^T)$.
With this definition $\Delta_K$ satisfies the following skein relation 
$$
\Delta_{L_+}-\Delta_{L_-}=(x^{-1}-x)\Delta_{L_0},
$$
where $L_+,L_-$ and $L_0$ are links that differ only in a small region as show in Figure \ref{skein}.  
 \begin{figure}[h!]
\centering 
  {\quad\qquad\includegraphics[scale=1]{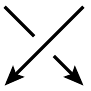}\quad}
  {\qquad\includegraphics[scale=1]{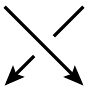}\quad}
  {\qquad\includegraphics[scale=1]{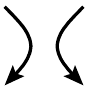}\quad}
  \caption{The crossings corresponding to $L_+,L_-$ and $L_0$, respectively.}
  \label{skein}
\end{figure}
Furthermore we can consider the Hermitian form $(1-\omega)S+(1-\overline\omega)S^T$ for any unit complex number $\omega$ \cite{Tristram}. 
The form is nonsingular when $\omega$ is not a root of the Alexander polynomial of $K$, and the signatures of these forms are called Tristram-Levine signatures. 
If we define $\sigma_\omega(K)$ to be the average of the one-sided limits of these signatures as we approach $\omega$, then $\sigma_\omega$ is a concordance invariant  for every
$\omega\in S^1$ (see, for example, \cite{Gor}), and in fact provide homomorphisms from $\mathcal{C}$ to $\mathbb{Z}$.
We use these signatures to prove the following result.

\begin{reptheorem}{thm:thm1}
Infinitely many of the knots $\{K_n\}$ are linearly independent in $\mathcal{C}$.
\end{reptheorem}
First we establish two technical lemmas. 
\begin{lemma}\label{lemma1}
 For each $n\in\mathbb{N}$ and each $\omega\in S^1\setminus\{1\}$ such that $\Delta_{K_0}(\omega^{\frac{1}{2}})\neq 0$ define
 $$
 F(\omega,n):=2(\textrm{Re}(\omega)-1)\left( 1-n+n\frac{\Delta_{K_1}(\omega^{\frac{1}{2}})}{\Delta_{K_0}(\omega^{\frac{1}{2}})} \right).
 $$
 We have
 \begin{itemize}
  \item if $F(\omega,n)<0$, then $\sigma_\omega(K_n) = \sigma_\omega(K_0)$;
  \item if $F(\omega,n)>0$, then $\sigma_\omega(K_n) = \sigma_\omega(K_0) \pm 2$.
 \end{itemize}
\end{lemma}

\begin{figure}
\includegraphics[scale=1.50]{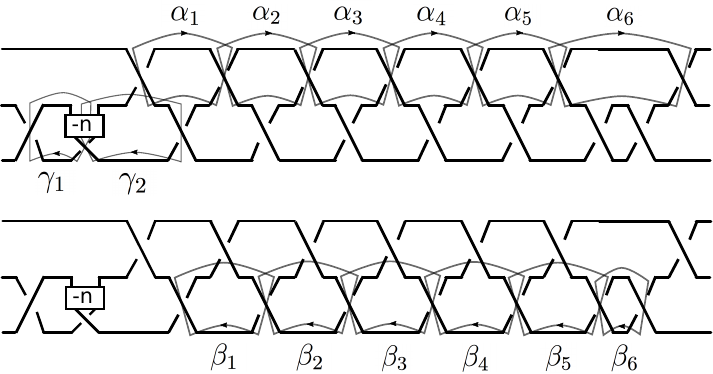}
\centering
 \caption{A basis for the first homology group of the Seifert surface of $K_n$.}
\label{basis}
\end{figure}

\begin{proof}
First note that since the Seifert form of a knot has even rank the expression $\Delta_{K}(\omega^{\frac{1}{2}})$ makes sense
even though $\omega^{\frac{1}{2}}$ is ambiguous.

 In Figure \ref{Knstretch} we see a natural Seifert surface for the knot $K_n$, consisting of three stacked disks and 16 twisted bands. 
 This is a genus 7 surface (non-minimal when $n=0$),
 and in Figure \ref{basis} we choose a basis for the first homology of the Seifert surface.
 Let $S_n$ denote the Seifert form for $K_n$ associated to this Seifert surface, oriented so that in Figure \ref{Knstretch} we are looking at the positive side of the disks.
 Choosing the ordered basis to be $\alpha_1,\cdots, \alpha_6, \beta_1,\cdots, \beta_6, \gamma_1, \gamma_2$, the corresponding matrix is
 $$\tiny{
 \left(\begin{array}{rrrrrrrrrrrrrr}
-1 & 0 & 0 & 0 & 0 & 0 & 1 & 0 & 0 & 0 & 0 & 0 & 0 & -1 \\
1 & -1 & 0 & 0 & 0 & 0 & -1 & 1 & 0 & 0 & 0 & 0 & 0 & 0 \\
0 & 1 & -1 & 0 & 0 & 0 & 0 & -1 & 1 & 0 & 0 & 0 & 0 & 0 \\
0 & 0 & 1 & -1 & 0 & 0 & 0 & 0 & -1 & 1 & 0 & 0 & 0 & 0 \\
0 & 0 & 0 & 1 & -1 & 0 & 0 & 0 & 0 & -1 & 1 & 0 & 0 & 0 \\
0 & 0 & 0 & 0 & 1 & -1 & 0 & 0 & 0 & 0 & -1 & 0 & 0 & 0 \\
0 & 0 & 0 & 0 & 0 & 0 & -1 & 0 & 0 & 0 & 0 & 0 & 0 & 0 \\
0 & 0 & 0 & 0 & 0 & 0 & 1 & -1 & 0 & 0 & 0 & 0 & 0 & 0 \\
0 & 0 & 0 & 0 & 0 & 0 & 0 & 1 & -1 & 0 & 0 & 0 & 0 & 0 \\
0 & 0 & 0 & 0 & 0 & 0 & 0 & 0 & 1 & -1 & 0 & 0 & 0 & 0 \\
0 & 0 & 0 & 0 & 0 & 0 & 0 & 0 & 0 & 1 & -1 & 0 & 0 & 0 \\
0 & 0 & 0 & 0 & 0 & 0 & 0 & 0 & 0 & 0 & 1 & -1 & 0 & 0 \\
0 & 0 & 0 & 0 & 0 & 0 & 0 & 0 & 0 & 0 & 0 & 0 & -n & n \\
0 & 0 & 0 & 0 & 0 & 0 & 0 & 0 & 0 & 0 & 0 & 0 & n + 1 & -n - 1
\end{array}\right).}
 $$
 Note that the upper-left $12\times 12$ minor coincides with the Seifert form for $K_0=T_{3,7}$ associated to the Seifert surface obtained from Figure \ref{Knstretch} after the obvious second Reidemeister move. 
 Let us indicate this form with $M$.
 We can write 
 $$
 \det((1-\omega)S_n+(1-\overline{\omega})S_n^T)=\lambda_1(\omega)\cdots\lambda_{12}(\omega)\cdot \mu_1(\omega,n)\cdot\mu_2(\omega,n),
 $$
 where $\lambda_1(\omega),\dots,\lambda_{12}(\omega)$ are the eigenvalues of the Hermitian form $(1-\omega)M+(1-\overline{\omega})M^T$, and 
 $\mu_1(\omega,n)$ and $\mu_2(\omega,n)$ are eigenvalues that depend on $n$. Since the signature $\sigma_{\omega}(K_n)$ is the sum of the signs of the eigenvalues we have
 $$
 \sigma_{\omega}(K_n)-\sigma_{\omega}(K_0)=\sign(\mu_1(\omega,n))+\sign(\mu_2(\omega,n)).
 $$
 Therefore $\sigma_{\omega}(K_n)=\sigma_{\omega}(K_0)$ if $\mu_1(\omega,n)$ and $\mu_2(\omega,n)$ have different signs,
 and $\sigma_\omega(K_n) = \sigma_\omega(K_0) \pm 2$ if $\mu_1(\omega,n)$ and $\mu_2(\omega,n)$ have the same sign.
 We conclude then by showing that $F(\omega,n)=\mu_1(\omega,n)\cdot\mu_2(\omega,n)$.
 
 We have
 
 \begin{align*}
  \det((1-\omega)S_n+(1-\overline{\omega})S_n^T)&=(1-\omega)^{14}\det(S_n+\frac{1-\overline{\omega}}{1-\omega}S_n^T)\\
  &=(1-\omega)^{14}\det(S_n-\overline{\omega}S_n^T)\\
  &=\overline{\omega}^7(1-\omega)^{14}\det(\overline{\omega}^{-\frac{1}{2}}S_n-\overline{\omega}^{\frac{1}{2}}S_n^T)\\
    &=\overline{\omega}^7(1-\omega)^{14}\det(\omega^{\frac{1}{2}}S_n-\omega^{-\frac{1}{2}}S_n^T)\\
    &=\overline{\omega}^7(1-\omega)^{14}\Delta_{K_n}(\omega^{\frac{1}{2}}).
 \end{align*}
 Similarly we get $\det((1-\omega)M+(1-\overline{\omega})M^T)=\overline{\omega}^6(1-\omega)^{12}\Delta_{K_0}(\omega^{\frac{1}{2}})$. Now the quantity
 $\mu_1(\omega,n)\cdot\mu_2(\omega,n)$ is the quotient of these two determinants and so we obtain
\begin{equation}\label{eq1}
 \mu_1(\omega,n)\cdot\mu_2(\omega,n)=\overline{\omega}(1-\omega)^2\frac{\Delta_{K_n}(\omega^{\frac{1}{2}})}{\Delta_{K_0}(\omega^{\frac{1}{2}})}.
\end{equation}

Here we use the skein relation as follows. In Figure \ref{ourknotskein} we have chosen a crossing so that $L_{-}$ (pictured) corresponds to $K_{n-1}$, $L_{+}$ corresponds to $K_n$,
and $L_0$ corresponds to a 2-component link $L$ that does not depend on $n$.

Applying the skein relation recursively we see that $\Delta_{K_n}(\omega^{\frac{1}{2}})=\Delta_{K_0}(\omega^{\frac{1}{2}})+n(\omega^{-\frac{1}{2}}-\omega^{\frac{1}{2}})\Delta_{L}(\omega^{\frac{1}{2}})$,
and furthermore $(\omega^{-\frac{1}{2}}-\omega^{\frac{1}{2}})\Delta_{L}(\omega^{\frac{1}{2}})=\Delta_{K_1}(\omega^{\frac{1}{2}})-\Delta_{K_0}(\omega^{\frac{1}{2}})$. Using
these two identities, from \eqref{eq1} we obtain
\begin{align*}
 \mu_1(\omega,n)\cdot\mu_2(\omega,n)&=\overline{\omega}(1-\omega)^2\left(1-n+n\frac{\Delta_{K_1}(\omega^{\frac{1}{2}})}{\Delta_{K_0}(\omega^{\frac{1}{2}})}\right)\\
 &=2(\textrm{Re}(\omega)-1)\left(1-n+n\frac{\Delta_{K_1}(\omega^{\frac{1}{2}})}{\Delta_{K_0}(\omega^{\frac{1}{2}})}\right)\\
 &=F(\omega,n).\\
\end{align*}
\end{proof}

\begin{figure}
\includegraphics[scale=3.0]{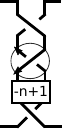}
\centering
 \caption{The chosen crossing for the skein relation.}
\label{ourknotskein}
\end{figure}

\begin{lemma}\label{lemma2}
 There exists $l_0>0$ and a sequence $\{\omega_l\}\subset S^1$ converging to $-1$ such that for every $l>l_0$ the following conditions are satisfied:
 \begin{itemize}
  \item $\sigma_{\omega_l}(K_m) = \sigma_{\omega_l}(K_0)\pm 2$ for all $m \geq l$;
   \item $\sigma_{\omega_l}(K_m) = \sigma_{\omega_l}(K_0)$ for all $l_0< m < l$.
 \end{itemize}
\end{lemma}
\begin{proof}
 Consider the function $g(t) = \frac{\Delta_{K_1}(e^\frac{it}{2})}{\Delta_{K_0}(e^\frac{it}{2})}$. 
 Explicit computations using the Seifert forms introduced earlier show that
 {\tiny
 \begin{align*}
\Delta_{K_0}(x)&=x^{12}-x^{10}+x^6-x^4+1-x^{-4}+x^{-6}-x^{-10}+x^{-12};\\
  \Delta_{K_1}(x)&=2x^{12}-3x^{10}+x^{8}+2x^6-3x^4+x^2+1+x^{-2}-3x^{-4}+2x^{-6}+x^{-8}-3x^{-10}+2x^{-12}.
 \end{align*}}
 Substituting $x=e^{\frac{it}{2}}$ and simplifying we get
 $$
 g(t)=\frac{4\cos 6t-6\cos 5t + 2\cos 4t +4\cos 3t -6\cos 2t +2\cos t + 1}{2\cos 6t -2\cos 5t +2\cos 3t -2\cos 2t +1}.
 $$
 One can check that $g(\pi)=1$, $g'(\pi)=0$, and $g''(\pi)<0$.
 It follows that there exists an $\varepsilon >0$ such that $g(t)$ is continuous and monotonically increasing on the interval $[\pi-\varepsilon, \pi]$.
 Therefore there exists an $l_0$ and a sequence $\{t_l\}_{l>l_0}\subset[\pi-\varepsilon, \pi]$ such that $g(t_l)=1-\frac{1}{l}$. 
 For each $l>l_0$ we choose $s_l$
 such that $t_{l-1}<s_l<t_l$ and define $\omega_l:=e^{is_l}$. Now assume that $m\geq l$, we have
 $$
 F(\omega_l,m)=2(\textrm{Re}(\omega_l)-1)(1-m+mg(s_l))>0
 $$
 since $g(s_l)\leq g(s_m)<g(t_m)=1-\frac{1}{m}$. By Lemma \ref{lemma1}, $\sigma_{\omega_l}(K_m) = \sigma_{\omega_l}(K_0)\pm 2$. Similarly if $l_0<m< l$ then $F(\omega_l,m)<0$ and,
 again by Lemma \ref{lemma1}, we conclude that $\sigma_{\omega_l}(K_m) = \sigma_{\omega_l}(K_0)$.
 

\end{proof}
Observe that we can choose the $\omega_l$ in Lemma \ref{lemma2} such that the Seifert forms 
$(1-\omega_l)S_m+(1-\overline{\omega}_l)S_m^T$ are nonsingular, and therefore we do not have to deal with one-sided limits
as in the definition of $\sigma_{\omega}$.
Now we are ready to prove Theorem \ref{thm:thm1}.
\begin{proof}[Proof of Theorem \ref{thm:thm1}]
We will show that the set $\{K_l\}_{l>l_0}$ is linearly independent in $\mathcal{C}$, where $l_0$ is the positive integer defined in Lemma \ref{lemma2}.
 Assume, by contradiction, that there exist nonzero $\alpha_1,\dots,\alpha_L$ such that the knot $\sharp_{i=1}^L \alpha_i K_{l_i}$ is slice, where each $l_i > l_0$ and are ordered such 
 that $l_i>l_{i-1}$. Recall that for torus knots the signatures $\sigma_{\omega}$ can be described explicitly (\cite{Lith}). In particular
 $\sigma_{\omega_l}(K_0)\neq 0$ for $l$ sufficiently large. Choose $h>l_{L}$ such that $\sigma_{\omega_{h}}(K_0)\neq 0$.
 By Lemma \ref{lemma2} we have
 $$
 0=\sigma_{\omega_{h}}(\sharp_{i=1}^L \alpha_i K_{l_i})=\sigma_{\omega_{h}}(K_0)\sum_{i=1}^L\alpha_i.
 $$
  It follows that $\alpha_1+\dots+\alpha_L=0$.
 Now set $k=l_{L}$. Again, by Lemma \ref{lemma2}, we have
 $$
 0=\sigma_{\omega_{k}}(\sharp_{i=1}^L \alpha_i K_{l_i})=\sigma_{\omega_k}(K_{l_L})\alpha_L+\sigma_{\omega_k}(K_0)\sum_{i=1}^{L-1}\alpha_i.
 $$
 The two equalities imply $\sigma_{\omega_k}(K_{l_L})=\sigma_{\omega_k}(K_0)$ which contradicts Lemma \ref{lemma2}.
\end{proof}

Theorem \ref{thm:thm1} together with Propositions \ref{balls} and \ref{covers} shows that $\ker(\beta)$ contains a free and  infinitely generated subgroup.
In order to get a $\mathbb{Z}^{\infty}$ \emph{summand} we need a slightly different family of knots.
\begin{proof}[Proof of Corollary \ref{thm:thm2}]
 Consider the family of homomorphisms $\{\frac{1}{2}\sigma_{\omega_l}\}_{l>l_0}$ as in Lemma \ref{lemma2}. This gives a homomorphism
 $$
 \overline{\sigma}:\ker(\beta)\longrightarrow \Z^{\infty} \ \ \ ,\ \ \ \ K\mapsto (\frac{1}{2}\sigma_{\omega_{l_0+1}}(K),\ldots,\frac{1}{2}\sigma_{\omega_{l_0+k}}(K),\ldots).
 $$
 In order to conclude the proof it suffices to show that $\overline{\sigma}$ is surjective. It would then follow that we have a short exact sequence 
 $0 \rightarrow \ker \overline{\sigma} \rightarrow \ker \beta \rightarrow \mathbb{Z}^\infty \rightarrow 0$, and since $\mathbb{Z}^\infty$ is free abelian the sequence splits. 
 For this purpose, consider the knots $J_n:=K_n-K_0$ with $n>l_0$. By Lemma \ref{lemma2}, we have 
 \begin{align*}
  \overline{\sigma}(J_n)&=\frac{1}{2}(\sigma_{\omega_{l_0+1}}(K_n)-\sigma_{\omega_{l_0+1}}(K_0),\dots,\sigma_{\omega_{l_0+k}}(K_n)-\sigma_{\omega_{l_0+k}}(K_0),\dots)\\
  &=\pm(\overbrace{1,\dots,1}^{n-l_0},0,0,\dots).
 \end{align*}
 Here we use the fact that $\sigma_{\omega_l}(K_n) = \sigma_{\omega_l}(K_0)\pm 2$ for all $l\leq n$, and that the sign of the $\pm2$ is constant for fixed $n$.
 Hence we can arrange that 
 $\overline{\sigma}(J_{n}\pm J_{n-1})$ maps to $\pm e_{n-l_0}$, where $e_{n-l_0}$ is the $(n-l_0)$-th element of the standard basis of $\Z^{\infty}$, and this concludes the proof.
\end{proof}

\end{section}

\bibliographystyle{amsalpha}
\bibliography{knotconcordanceandhomologyspheregroups.bib}

\end{document}